\documentclass[11pt]{amsart}

\usepackage{amssymb,amsthm,amsmath, graphicx, esvect}
\usepackage{hyperref}
\usepackage[alphabetic]{amsrefs}
\usepackage[shortlabels]{enumitem}
\usepackage[margin=1.5in]{geometry}
\usepackage{tikz-cd} 

\newtheorem{theorem}{Theorem}[section]
\newtheorem{corollary}[theorem]{Corollary}
\newtheorem{lemma}[theorem]{Lemma}

\theoremstyle{remark}

\newcommand{\N}{\mathbb{N}}

\newcommand{\R}{\mathbb{R}}
\newcommand{\C}{\mathbb{C}}

\newcommand{\GL}{\mathrm{GL}}
\newcommand{\U}{\mathrm{U}}

\newcommand{\SU}{\mathrm{SU}}
\newcommand{\su}{\mathfrak{su}}
\newcommand{\g}{\mathfrak{g}}

\renewcommand{\epsilon}{\varepsilon}
\renewcommand{\leq}{\leqslant}
\renewcommand{\geq}{\geqslant}

\title{Distinguishing C*-algebras by their unitary groups}

\author{Lionel Fogang Takoutsing}
\address{L. Fogang Takoutsing, Department of Mathematics, University of Louisiana at Lafayette, Lafayette, LA 70504-1010 USA}

\author{Leonel Robert}
\address{L. Robert, Department of Mathematics, University of Louisiana at Lafayette, Lafayette, LA 70504-1010 USA}

\begin{document}
\begin{abstract}
We obtain partial affirmative answers to the question whether isomorphism of the unitary groups of two C*-algebras, either as topological groups or as discrete groups, implies isomorphism of the C*-algebras as real C*-algebras.  
\end{abstract}

\maketitle

\section{Introduction}
We investigate whether two unital C*-algebras with isomorphic unitary groups must be isomorphic as real C*-algebras. The converse is always true, as the real C*-algebra structure of a C*-algebra is sufficient to define its unitary group. This question was first addressed by Dye \cite{dye} in the context of von Neumann algebra factors. Al-Rawashdeh, Booth, and Giordano further studied this question in \cite{giordano}, using results from the classification program for simple nuclear C*-algebras to obtain the desired isomorphism at the C*-algebra level. Very recently, Sarkowicz followed a similar strategy in \cite{sarkowicz}, relying on up-to-date classification results. Chand and the second author approached the question in \cite{chand-robert} using the Lie group-Lie algebra correspondence and structure theorems for Lie homomorphisms between C*-algebras. However, the result in \cite{chand-robert} was confined to traceless C*-algebras. In this paper, we extend this approach to a broader class of C*-algebras.

Let $A$ be a unital C*-algebra. Denote by $\U(A)$ the group of unitaries in $A$ and by $\U_0(A)$ the connected component of 1 in $\U(A)$. Both groups are considered as topological groups equipped with the topology induced by the norm on $A$.

In addition  to $\U(A)$ and $\U_0(A)$, we consider the group $\SU_0(A)$ defined as the subgroup of $\U_0(A)$ generated by the set  $\{e^{h}:h\in \su(A)\}$. Here $\su(A):=\overline{[iA_{sa},iA_{sa}]}$ denotes the closed linear span of commutators of  skewadjoint elements of $A$.
The topology on $\SU_0(A)$ is such that for a sufficiently small $\epsilon>0$ the exponential map $\su(A)\ni h\mapsto e^h\in \SU_0(A)$ is a
homeomorphism from $\su(A)\cap B_\epsilon(0)$ to an open neighborhood of 1 in $\SU_0(A)$. This topology may in general defer from the topology induced by the norm.

\begin{theorem}\label{topisothm}
Let $A$ and $B$ be unital prime  $C^*$-algebras without 1-dimensional or 2-dimensional representations. The following are equivalent:
\begin{enumerate}[(i)]
\item
$\U(A)$ and $\U(B)$ are isomorphic as topological groups.
\item
$\U_0(A)$ and $\U_0(B)$ are isomorphic as topological groups.
\item
$\SU_0(A)$ and $\SU_0(B)$ are isomorphic as topological groups.
\item
$A$ and $B$ are isomorphic as real C*-algebras.
\end{enumerate} 
Moreover, if $\alpha\colon \SU_0(A)\to \SU_0(B)$ denotes the group isomorphism in (iii), then the isomorphism $\phi\colon A\to B$ in (iv) can be chosen such that it extends $\alpha$.
\end{theorem}
By an isomorphism of $A$ and $B$ as real C*-algebras we understand an $\mathbb{R}$-algebra isomorphism preserving the involution (and, automatically, the norm).
In the context of Theorem \ref{topisothm}, where the C*-algebras are prime, such a map must be either linear (over $\C$) or conjugate linear (see Lemma \ref{BorBop}). In the latter case, $x\mapsto \phi^*(x)$ is a C*-algebra isomorphism.  It follows that condition (iv) can be rephrased as ``$A$ is isomorphic to $B$ or to its opposite algebra $B^{op}$, as C*-algebras".

The technique of proof for Theorem \ref{topisothm}  follows a well-trodden path: We regard $\U_0(A)$ and $\SU_0(A)$ as Banach-Lie groups with Lie algebras
$iA_{sa}$ (the skewadjoint elements of $A$) and $\su(A)$. From the topological groups isomorphisms we derive  an isomorphism of the corresponding  Lie algebras. 
Then, using powerful results on the structure of  Lie algebra homomorphisms from \cite{bresarbook}, we derive the existence of a real C*-algebra isomorphism.  It is in this second step that the hypotheses of primality and absence of 1 or 2-dimensional representations get used.
 
The inclusion of $\SU_0(A)$ in our analysis proves to be crucial in obtaining affirmative answers to the isomorphism question without assuming a topological isomorphism at the group level.

Let $[A,A]$ denote the linear span of the elements $[x,y]:=xy-yx$, with $x,y\in A$, i.e., the commutators in $A$. Let $\overline{[A,A]}$ denote the norm closure of $[A,A]$. Let  us say that the C*-algebra $A$ has  \emph{bounded commutators generation} (BCG) if there exist $N\in \N$ and $C>0$ such that for all $x\in\overline{[A,A]}$, we have
\[
x=\sum_{i=1}^N [a_i,b_i]
\]
for some $a_i,b_i\in A$ such that $\|a_i\|\cdot \|b_i\|\leq C\|x\|$ for all $i$.

An element $x$ of a C*-algebra is called \emph{full} if it generates the C*-algebras as a closed two-sided ideal, and square-zero if $x^2=0$.
It is shown in \cite{chand-robert} that if $A$ has BCG and a full square-zero element, then $\SU_0(A)$ has the invariant automatic continuity property (recalled in the next section). 
Both the existence of a full square-zero element and the BCG property are not  uncommon among well-studied classes of  C*-algebras. 
For example, exact C*-algebras that  tensorially absorb the Jiang-Su algebra, and more generally, exact C*-algebras whose Cuntz semigroup is almost unperforated and almost divisible, have both properties (\cite{ng-robert}). Traceless unital C*-algebras have BCG by \cite{pop}.

\begin{theorem}\label{BCGisothm}
Let $A$ and $B$ be separable unital prime $C^*$-algebras without 1-dimensional or 2-dimensional representations. Suppose also that $A$ and $B$ both have BCG and a full square-zero element. The following are equivalent:
\begin{enumerate}[(i)]
\item
$\U(A)$ and $\U(B)$ are isomorphic as  groups.
\item
$\U_0(A)$ and $\U_0(B)$ are isomorphic as  groups.
\item
$\SU_0(A)$ and $\SU_0(B)$ are isomorphic as  groups.
\item
$A$ and $B$ are isomorphic as real C*-algebras.
\end{enumerate}
\end{theorem}  
To prove Theorem \ref{BCGisothm} we exploit the invariant automatic continuity of $\SU_0(A)$ and invoke Theorem \ref{topisothm}. 
The implication (ii) implies (iv) in the case that the C*-algebras are assumed to be traceless  was obtained in \cite{chand-robert}. 


It seems likely that the equivalence of  (i)-(iv) in  Theorems \ref{topisothm} and \ref{BCGisothm} holds for larger classes of C*-algebras than those covered by  
these theorems. The authors are not aware of an example of two C*-algebras with isomorphic unitary groups that are not isomorphic as real C*-algebras.

\emph{Notation conventions}: Given subsets $X,Y$ of $A$, we denote by $X\cdot Y$ the additive group generated by the products $xy$ with $x\in X$ and $y\in Y$. We denote by $[X,Y]$ the additive group generated by the commutators $[x,y]$, with $x\in X$ and $y\in Y$. In all cases where we  use this notation, the sets  $X$ and $Y$ are closed under scalar multiplication by $\R$, so that $X\cdot Y$ and  $[X,Y]$ are also  vector subspaces.

\subsection*{Acknowledgements} We thank Pawel Sarkowicz for feedback on this note.

\section{From groups to Lie algebras to C*-algebras}
Let $A$ be a unital Banach algebra over $\R$. Let $\GL(A)$ denote the group of invertible elements of $A$, which we regard as a Banach-Lie group with Lie algebra $A$
and exponential map $A\ni a\mapsto e^a\in \GL(A)$. 

We review some facts about \emph{analytic subgroups} of $\GL(A)$, in the sense of \cite{maissen}. We refer the reader to \cite{neeb2004}*{Remark IV.4} for a discussion of various notions of Banach-Lie subgroup of a Banach-Lie group.

Let $\g\subseteq A$  be a norm-closed real vector subspace of $A$ such that $[\g,\g]\subseteq \g$, henceforth called a closed Lie subalgebra of $A$.
Denote by $G_{\g}$ the subgroup of $\GL(A)$ generated by the set $\{e^l:l\in \g\}$.  By \cite{hofmann-morris}*{Theorem 5.52 (i)} there is a unique connected group
topology on $G_{\g}$ such that the exponential map $\g\ni h\mapsto e^h\in G_{\g}$ is a homeomorphism from some neighborhood of $0$ in $\g$ to a neighborhood of 1 in $G_{\g}$.
Then $G_{\g}$ is a Banach-Lie group whose Lie algebra $\mathcal L(G):=\mathrm{Hom}(\R,G_{\g})$  is isomorphic to $\g$ as a completely normable Lie algebra. Upon identifying $\mathcal L(G)$ with $\g$, the exponential map takes the form $\g\ni l\mapsto e^l\in G_{\g}$.  We will make use of the functoriality of the Lie algebra of a Lie group in the context of the Lie subalgebras $\g$ and their Banach-Lie groups $G_{\g}$: 


\begin{theorem}
Let $\g_1 \subseteq A $ and $ \g_2 \subseteq B $ be closed Lie subalgebras of real Banach algebras $A$ and $B$. If 
$\alpha\colon G_{\g_1} \to G_{\g_2}$ is continuous group homomorphism, then there exists a unique bounded Lie algebra homomorphism $\psi\colon \g_1 \to \g_2 $  such that $\alpha(e^h) = e^{\psi(h)} $ for all $h\in \g_1$.
\end{theorem}
\begin{proof}(Sketch)
To each $h\in \g_1$ we associate the unique $\psi(h)\in B$ such that $\alpha(e^{th})=e^{t\psi(h)}$ for all $t\in \R$. The continuity of $\alpha$ readily implies that  $\psi(h)\in \g_2$.  We use the standard arguments for recovery of addition, scalar multiplication, and the Lie product, to show that $\psi$ is a Lie homomorphism (see \cite{neeb2004}*{Theorem IV.2}). The continuity of $\psi$ at 0 follows from the fact that $\psi(h)=\log(\alpha(e^h))$ for all $h\in\g_1$ in a sufficiently small neighborhood of 0. 
\end{proof}

Suppose now that  $A$ is a unital C*-algebra.  Let $A_{sa}$ denote the set of selfadjoint elements of $A$. Recall that we denote by
$[A,A]$ the linear span of the   the commutators in $A$ and by  $\overline{[A,A]}$ the norm closure of $[A,A]$. 
Define
\[
\su(A):=\overline{[iA_{sa},iA_{sa}]}=\overline{[A_{sa},A_{sa}]}=\overline{[A,A]} \cap iA_{sa}.
\]
\begin{enumerate}
\item
If  we choose the Lie subalgebra $\g=iA_{sa}$, then $G_\g=\U_0(A)$, and the topology on $G_\g$ is simply the norm topology.  
\item
If we choose the Lie subalgebra $\g=\su(A)$, then $G_\g=\SU_0(A)$, by the very definition of $\SU_0(A)$. The topology on $\SU_0(A)$ is finer than, and may in general be different from, the norm topology (see \cite{chand-robert}*{Remark 5.4}).
\end{enumerate}

We immediately deduce from the previous theorem the following corollary: 

\begin{corollary}\label{fromalpha2psi}
Let $A$ and $B$ be unital C*-algebras. 
\begin{enumerate}[(i)]
\item
Let $\alpha\colon \U_0(A)\to \U_0(B)$ be a continuous group homomorphism. Then there exists a unique bounded Lie algebra homomorphism $\psi\colon iA_{sa}\to iB_{sa}$
such that $\alpha(e^h)=e^{\psi(h)}$ for all $h\in iA_{sa}$.
\item
Let $\alpha\colon \SU_0(A)\to \SU_0(B)$ be a continuous group homomorphism. Then there exists a unique bounded Lie algebra homomorphism $\psi\colon \su(A)\to \su(B)$
such that $\alpha(e^h)=e^{\psi(h)}$ for all $h\in \su(A)$.
\end{enumerate}
\end{corollary}

Next we examine the Lie homomorphism obtained in Corollary \ref{fromalpha2psi} (ii). Our main tool is 
\cite{bresarbook}*{Corollary 6.20}, which we reproduce here with minor 
changes of notation and wording. For a subset $X$ of a ring, we use the notation $\langle X\rangle$ for the ring generated by $X$.

\begin{theorem}[\cite{bresarbook}*{Corollary 6.20}]\label{bresar}
Let $S$ be a Lie ideal of the Lie ring $L$ of skew elements of a ring with involution $A$.  Let $B$ be a prime ring with involution. Denote by $C$ the extended centroid of $B$ and by
$K$ be the skewadjoint elements of $B$. Let $R$ be a noncentral Lie ideal of $K$. Suppose that $S$ admits the operator $\frac12$, 
and suppose that $\mathrm{deg}(B) \geq 21$ and $\mathrm{char}(B) \neq  2$. If $\psi\colon S\to R/(R \cap C)$ is an onto Lie homomorphism, then there 
exists a ring homomorphism $\phi\colon\langle S\rangle\to\langle R\rangle C+C$ such that $\psi(x)=q(\phi(x))$ for all $x\in S$, where $q\colon B\to B/B\cap C$ denotes the quotient map.
\end{theorem}

\emph{Note}: The  ring homomorphism $\phi$ from the theorem automatically preserves the involution. This follows from the fact that the elements
of $\langle S\rangle$ are sums of finite products of elements of $S$, and from the calculation
\begin{align*}
\phi((s_1\cdots s_n)^*) &=(-1)^n\phi(s_n\cdots s_1)\\
&=(\phi(s_n))^*\cdots(\phi(s_1))^*
=\Big(\phi(s_1)\cdots \phi(s_n)\Big)^*=\Big(\phi(s_1\cdots s_n)\Big)^*.
\end{align*}

We shall apply  Theorem \ref{bresar} in the context of Corollary \ref{fromalpha2psi} (ii), simplifying some of the hypotheses in the process.

\begin{lemma}\label{suring}
Let $A$ be a unital C*-algebra without 1-dimensional or 2-dimensional representations. Then 
\[
A= [A_{sa},A_{sa}]^2 + [A_{sa},A_{sa}]^3 +  [A_{sa},A_{sa}]^4.
\]
\end{lemma}
\emph{Note}: Recall that, for sets $X,Y\subseteq A$, we denote by $X\cdot Y$ the additive group generated by $\{xy:x\in X,\, y\in Y\}$.
\begin{proof}
Given  $a,b\in A$, we denote by $a\circ b$ their Jordan product; $a\circ b=ab+ba$.

Let $M$ denote the additive group   generated by  $\{a\circ b: a,b\in [A_{sa},A_{sa}]\}$. Note that $M$ is an $\R$-vector subspace of $A_{sa}$
and that it is also the additive group generated by the set $\{a^2:a\in [A_{sa},A_{sa}]\}$. 

Let us show that $L:=M+iM$ is a Lie ideal of $A$, i.e., $[L,A]\subseteq L$. To this end, we write $L=M+iM$ and $A=A_{sa}+iA_{sa}$ in $[L,A]$. It is then clear that it is sufficient to show that $[M,iA_{sa}]\subseteq M$.  But this inclusion follows from the calculation  $[a^2,b]=a\circ [a,b]$, for if  $a\in [A_{sa},A_{sa}]$ and $b\in iA_{sa}$, then $a\circ [a,b]$ is
the Jordan product of two elements in $[A_{sa},A_{sa}]=[iA_{sa},iA_{sa}]$.
 
Let us show that $L$ is \emph{fully noncentral}, i.e., that it is not mapped to the center of any nonzero quotient of $A$.  Suppose for the sake of contradiction that $L$ is mapped to the center of a nonzero quotient $A/I$, which we may assume  to be a simple quotient enlarging $I$ to a maximal ideal if necessary. Since  $[A_{sa},A_{sa}]$ is mapped onto $[(A/I)_{sa},(A/I)_{sa}]$,  the set  
\[
\{a\circ b: a,b\in [(A/I)_{sa},(A/I)_{sa}]\}
\] 
is contained in the center of $A/I$. By assumption, every representation of $A/I$ has dimension at least 3. Hence, by Glimm's lemma, there exists a non-zero homomorphism $\phi\colon M_3(\C)\otimes C_0(0,1]\to A/I$ (\cite[Proposition 3.10]{robert-rordam}). 
Setting $x_1=\phi(e_{21}\otimes t)$ and $x_2=\phi(e_{31}\otimes t)$, we get $x_1,x_2\in A/I$  and nonzero pairwise orthogonal positive elements $a,b,c\in A/I$ such that 
\[
a=x_1^*x_1=x_2^*x_2, \, b=x_1x_1^*,\, c=x_2x_2^*. 
\]
Observe that 
\[
a-b=[x_1^*,x_1] \in [A/I,A/I]\cap A_{sa}=i[(A/I)_{sa},(A/I)_{sa}],
\] 
and similarly $a-c\in i[(A/I)_{sa},(A/I)_{sa}]$. Hence, $2a^2=(a-b)\circ (a-c)$ is a central element of $A/I$, and by functional calculus $a$ is also central. Similarly, $b$ and $c$ are central. But, since $A/I$ is simple, its center is $\C\cdot 1$,  and thus cannot contain three pairwise orthogonal nonzero elements. This is the desired contradiction.

By  \cite[Theorem 3.1 (ii)]{robertlie}, a fully non-central Lie ideal in a unital C*-algebra must contain $[A,A]$. Thus,  
$[A,A]\subseteq L$. Comparing the selfadjoint parts of these two sets we get that $i[A_{sa},iA_{sa}] \subseteq  M$.
Using now that $[A,A] =  i[A_{sa},A_{sa}] + [A_{sa},A_{sa}]$, we get that  
\[
[A,A] \subseteq M+[A_{sa},A_{sa}] \subseteq [A_{sa},A_{sa}]^2 + [A_{sa},A_{sa}].
\]
Since $A$ is unital and has no 1-dimensional representations,  $A=[A,A]^2$  by \cite[Theorem 4.3]{gardella-thiel} (cf. \cite[Theorem 3.2]{robertlie}). 
Squaring both sides of the inclusion $[A,A]\subseteq [A_{sa},A_{sa}]^2 + [A_{sa},A_{sa}]$ the lemma readily follows. 
\end{proof}

We call a map $\phi\colon A\to B$ between C*-algebras an $\R^*$-homomorphism if it is an $\R$-algebra homomorphism that preserves the involution.
Equivalently, it suffices to require that  $\phi$ be additive, multiplicative and involution preserving, as in this case $\R$-linearity follows automatically.

\begin{theorem}\label{surjectivelie}
Let $A$ and $B$ be unital C*-algebras. Suppose that $A$ has no 1 or 2-dimensional representations and $B$ is prime and infinite dimensional.
Then any surjective continuous Lie algebra homomorphism $\psi \colon \su(A)\to \su(B)$ has a unique extension to an $\R^*$-homomorphism 
$\phi\colon A\to B$.
\end{theorem}

\begin{proof} 
Since, by the previous lemma, $\su(A)$ generates $A$ as a ring,  it is clear that $\phi$ is unique. To prove its existence,
we apply Theorem \ref{bresar} with $S=\su(A)$, which is  Lie ideal of the Lie ring $L=iA_{sa}$ of skewadjoint elements of $A$. It is clear that this choice of $S$ ``admits the operator $1/2$", since it is a vector subspace.  By the previous lemma, the ring generated by $\su(A)$ is equal to $A$.

On the codomain side, we set $R=\su(B)$, which is a Lie ideal of  $iB_{sa}$. 
This Lie ideal is non-central unless $B$ is commutative, a possibility that we have ruled out by assumption. (Proof: If $a,b\in B_{sa}$ are such that $[a,b]$ commutes with $a$, then 
$[a,b]$ is quasinilpotent, by the Kleinecke-Shirokov Theorem \cite{kleinecke}. Since $[a,b]$ is normal, we get that $[a,b]=0$. Applied to every pair $a,b\in B_{sa}$, this  shows that if $\su(B)$ is contained in the center then $B$ is commutative.) 

The extended centroid of a prime C*-algebra is $\C\cdot 1$, by \cite[Corollary 2.4]{ara}. Thus, $C=\C\cdot 1$.

In Theorem \ref{bresar}, $\deg(B)$ denotes the algebraicity degree of $B$ over its extended centroid, i.e., the least $n$ such that every element satisfies a polynomial equation of degree $n$  over $C=\C\cdot 1$. Since we have assumed  $B$ to be infinite dimensional, we have  $\deg(B)=\infty\geq 21$ (see \cite[Theorem C.2]{bresarbook}).

Let $q\colon B\to B/\C\cdot 1$ denote  the quotient map dividing by the center of $B$. Then $\bar\psi=q\circ \psi$ is a Lie homomorphism from   $\su(A)$ onto $\su(B)/(\su(B)\cap \C\cdot 1)$ (identified with the image of $\su(B)$ under $q$). By Theorem \ref{bresar}, there exists a ring homomorphism $\phi\colon A\to B$  such that $q(\phi(x))=\bar\psi(x)$ for all $x\in \su(A)$.  By the remark after Theorem \ref{bresar}, $\phi$ preserves the involution. It follows that $\phi$ is an $\R^*$-homomorphism.

Let us show that $\phi$ extends $\psi$. Define $\lambda(x):=\phi(x)-\psi(x)$ for all $x\in \su(A)$. Observe that $\lambda$ is an additive map
taking values in $\C\cdot 1$. For $a,b\in \su(A)$ we have that
\begin{align*}
\lambda([a,b]) &=[\phi(a),\phi(b)]-[\psi(a),\psi(b)]\\
&=[\lambda(a)+\psi(a),\lambda(b)+\psi(b)]-[\psi(a),\psi(b)]=0.
\end{align*}
Thus, $\lambda$ vanishes on $[\su(A),\su(A)]$, and by continuity, on $\overline{[\su(A),\su(A)]}$. To finalize the proof, it will suffice to show that the Lie algebra $\su(A)$ is topologically perfect, i.e., $\su(A)=\overline{[\su(A),\su(A)]}$, as then 
 $\lambda=0$, showing that $\phi$ extends $\psi$. 

In any C*-algebra we have that 
\[
\overline{[A,A]} = \overline{[[A,A],[A,A]]}
\] 
(\cite[Corollary 2.8]{robertlie}). Intersecting both sides with $iA_{sa}$ and repeatedly using that  $[A,A]=i[A_{sa},A_{sa}]+[A_{sa},A_{sa}]$ we get that 
 \[
 \su(A)=\overline{[A,A]}\cap iA_{sa}=\overline{[[A,A],[A,A]]}\cap iA_{sa}\subseteq \overline{[\su(A),\su(A)]}.
 \]
 Thus, $\su(A)$ is topologically perfect.
\end{proof}

Combining Theorem \ref{surjectivelie} and Corollary \ref{fromalpha2psi} we get:

\begin{corollary}\label{surjectivealpha}
Let $A$ and $B$ be separable and as in Theorem \ref{surjectivelie}. Then any surjective homomorphism of topological groups $\alpha\colon \SU_0(A)\to \SU_0(B)$ 
extends uniquely to a surjective $\R^*$-homomorphism $\phi\colon A\to B$.
\end{corollary}
\begin{proof}
By Corollary \ref{fromalpha2psi} (ii),   there exists a unique $\psi\colon \su(A)\to \su(B)$, bounded  Lie homomorphism,  such that $\alpha(e^h)=e^{\psi(h)}$ for all $h\in \su(A)$. 
Let us show that $\psi$ is surjective. Since $\alpha$ is a surjective group homomorphism between polish groups, it is open (\cite[Theorem 1.5]{hof-morris-survey}). 
Let $\epsilon>0$ be  such that if  $z\in \su(B)$ and $\|z\|<\epsilon$, then  there exists $h\in \su(A)$ such that $\alpha(e^h)=e^z$ and $\|h\|<1/\|\psi\|$. Then
\[
e^{\psi(h)}=\alpha(e^h)=e^z.
\] 
Since  $\|\psi(h)\|\leq 1$, taking the logarithm of both sides we get  $\psi(h)=z$. Thus, $B_\epsilon(0)\cap \su(B)$ is contained in the range of $\psi$, showing that it is surjective. By Theorem \ref{surjectivelie}, $\psi$ has a unique extension to an $\R^*$-homomorphism $\phi\colon A\to B$. Then 
\[
\phi(e^h)=e^{\phi(h)}=e^{\psi(h)}=\alpha(e^h),
\]
for all $h\in \su(A)$. Since $\{e^h:h\in \su(A)\}$ generates $\SU_0(A)$, we get that $\phi$ extends $\alpha$. 
\end{proof}

Let us  prove the first theorem stated in the introduction:

\begin{proof}[Proof of Theorem \ref{topisothm}]
The restriction of an $\R^*$-isomorphism $\phi\colon A\to B$ to $\U(A)$, $\U_0(A)$, and $\SU_0(A)$ is readily seen to give an isomorphism of topological groups
with $\U(B)$, $\U_0(B)$, and $\SU_0(B)$, respectively. Thus, (iv) implies (i), (ii), and (iii).

(iii) implies (iv): Let  $\alpha \colon \SU_0(A)\to \SU_0(B)$ be a topological groups isomorphism. By Corollary \ref{fromalpha2psi} (ii),  $\alpha$ gives rise to a continuous 
Lie algebras isomorphism  $\psi\colon \su(A)\to \su(B)$.

Suppose first that both  $A$ and $B$ are infinite dimensional. Then, by Theorem \ref{surjectivelie},   there exists a unique $\R^*$-homomorphism $\phi\colon A\to B$ that extends  $\psi$.   Similarly $\psi^{-1}\colon \su(B)\to \su(A)$ has  a unique extension to an $\R^*$-homomorphism $\phi'\colon B\to A$. Since $\phi'\phi$ is the identity on $\su(A)$, and $\su(A)$ generates $A$ as a ring (Lemma \ref{suring}),  $\phi'\phi$ is the identity on $A$. Similarly, $\phi\phi'$ is the identity on $B$. Hence, $\phi$ is an $\R^*$-isomorphism from $A$ to $B$. As in the proof of Corollary \ref{surjectivealpha}, we deduce that $\phi$ extends $\alpha$. This proves (iv) and the assertion that the $\R^*$-homomorphism $\phi$ extends $\alpha$.

Suppose now that either $A$ or $B$ is finite dimensional. Let's say $A$ is finite dimensional. By Theorem \ref{fromalpha2psi}, from the topological groups isomorphism  $\SU_0(A)\cong \SU_0(B)$ we deduce that $\su(A)\cong \su(B)$ as Lie algebras. In particular, $\su(B)$ is a finite dimensional vector space.  Then, by Lemma \ref{suring}, $B$
is finite dimensional. Since $A$ and $B$ are both prime and finite dimensional, we have that  $A\cong M_m(\C)$ and $B\cong M_n(C)$ form some $m,n\in \N$. 
Comparing the vector space dimensions of $\su(A)\cong\su(m)$ and $\su(B)\cong \su(n)$ we deduce that  $m=n$. Thus, $A\cong B\cong M_n(\C)$ for some $n\geq 3$.

It remains to show that every (continuous) automorphism $\alpha\colon \SU(n)\to \SU(n)$ has an extension to an $\R^*$-automorphism $\phi\colon M_n(\C)\to M_n(\C)$. Indeed, in this case either $\alpha$ is inner or $x\mapsto \overline{\alpha(x)}$ is inner, as follows from the calculation of the group of automorphisms of $\su(n)$ (\cite[Proposition D.40]{fulton-harris}, \cite{mathoverflow}).

(i) implies (ii): This is straightforward since a topological groups isomorphism from   $\U(A)$ to $\U(B)$ maps  $\U_0(A)$ bijectively onto $\U_0(B)$.

(ii) implies (iv): By Corollary \ref{fromalpha2psi} (i), an isomorphism of $\U_0(A)$ with $\U_0(A)$  gives rise to a bounded Lie algebras isomorphism $\psi\colon iA_{sa}\to iB_{sa}$.  
Then $\psi$ restricts to an isomorphism from $\su(A)$ to $\su(B)$. As argued in the proof of (iii) implies (iv), $\psi|_{\su(A)}$ extends to an isomorphism of $A$ and $B$ as real C*-algebras. 
\end{proof}

The reformulation of Theorem \ref{topisothm} (iv) as $A\cong B$ or $A\cong B^{op}$ as C*-algebras follows from the following lemma:

\begin{lemma}\label{BorBop}
Let $A$ be a unital C*-algebra and $\phi\colon A\to B$ an $\R^*$-homomorphism onto a C*-algebra $B$ with center $\C\cdot 1$ (in particular, a unital prime C*-algebra). 
Then $\phi$ is either $\C$-linear or conjugate $\C$-linear. In the latter case, $x\mapsto \phi(x)^*$ is $\C$-linear.
\end{lemma}
\begin{proof}
Consider $u=\phi(i\cdot 1)$. This is a central element in $B$ such that $u^2=-1$. Since the center of $B$ is $\C \cdot 1$, either
$u=i1$ or $u=-i1$. In the first case we deduce that $\phi$ is $\C$-linear, and in the second that it is conjugate $\C$-linear.
\end{proof}

We  now proceed to prove the second theorem stated in the introduction. Let us recall  the definition of the invariant automatic continuity property,  defined in \cite{dowerk-thom}. A topological group $G$ has the invariant automatic continuity property if every group homomorphism $\alpha\colon G \to H$, where $H$ is a topological separable SIN group, is continuous. Here, a SIN group, or ``small invariant neighborhoods" group, is a group that has a basis of neighborhoods of the identity which remain invariant under conjugation.

By \cite[Theorem B]{chand-robert}, if a unital C*-algebra $A$ has the bounded commutators generation (BCG) property and a square-zero element, then $\SU_0(A)$ has  the invariant automatic continuity property.

\begin{proof}[Proof of Theorem \ref{BCGisothm}]
We have already remarked on the fact that  (iv) implies (i), (ii), (iii).

(iii) implies (iv): Let $\alpha\colon \SU_0(A)\to \SU_0(B)$ be a group isomorphism.  By \cite[Theorem B]{chand-robert},  
$\SU_0(A)$ and $\SU_0(B)$ are SIN groups with the invariant automatic continuity property.  Thus, $\alpha$ is a topological groups isomorphism. It follows from Theorem \ref{topisothm} that  $A$ and $B$ are isomorphic as real C*-algebras. 


(ii) implies (iii): Let $\alpha\colon \U_0(A)\cong \U_0(B)$ be a group isomorphism. Then $\alpha$ restricts to an isomorphism between the commutators subgroups $\mathrm{DU}_0(A)$ and $\mathrm{DU}_0(B)$. The assumption that $A$ has the BCG property implies that $[A,A]=\overline{[A,A]}$, and this in turn implies that $\SU_0(A)=\mathrm{DU}_0(A)$, since $\{e^h:h\in [A_{sa},A_{sa}]\}$ is a generating set for $\mathrm{DU}_0(A)$ (\cite[Theorem 6.2]{robertnormal}). Similarly,  $\SU_0(B)=\mathrm{DU}_0(B)$. We thus get that $\SU_0(A)\cong \SU_0(B)$.

(i) implies (iii): Let $\alpha\colon \U(A)\to \U(B)$ be a group isomorphism. Since $\SU_0(A)$ has the invariant automatic continuity property, the restriction of $\alpha$ to $\SU_0(A)$,  regarded as a group homomorphism from  $\SU_0(A)$  to $\U(A)$, is continuous. Since $\SU_0(A)$ is connected, 
$\alpha$ maps $\SU_0(A)$ into $\U_0(B)$. As argued in the preceding paragraph,  $\SU_0(A)$ agrees with the commutator subgroup $\mathrm{DU}_0(A)$ under the assumptions of the theorem. Moreover,  $\mathrm{DU}_0(A)$ is a perfect group, by \cite[Theorem 6.2]{robertnormal}. Hence, $\alpha$ must map $\SU_0(A)$ into $\mathrm{DU}_0(B)=\SU_0(B)$. Arguing symmetrically, $\alpha^{-1}$ maps $\SU_0(B)$ to $\SU_0(A)$. Thus, $\SU_0(A)\cong \SU_0(B)$.
\end{proof}


\begin{bibdiv}
\begin{biblist}

\bib{giordano}{article}{
   author={Al-Rawashdeh, A.},
   author={Booth, A.},
   author={Giordano, T.},
   title={Unitary groups as a complete invariant},
   journal={J. Funct. Anal.},
   volume={262},
   date={2012},
   number={11},
   pages={4711--4730},
}

\bib{ara}{article}{
   author={Ara, P.},
   title={The extended centroid of $C^*$-algebras},
   journal={Arch. Math. (Basel)},
   volume={54},
   date={1990},
   number={4},
   pages={358--364},
}



\bib{bresarbook}{book}{
   author={Bre\v{s}ar, M.},
   author={Chebotar, M. A.},
   author={Martindale, W. S., III},
   title={Functional identities},
   series={Frontiers in Mathematics},
   publisher={Birkh\"{a}user Verlag, Basel},
   date={2007},
   pages={xii+272},
   isbn={978-3-7643-7795-3},
}

\bib{chand-robert}{article}{
   author={Chand, A.},
   author={Robert, L.},
   title={Simplicity, bounded normal generation, and automatic continuity of
   groups of unitaries},
   journal={Adv. Math.},
   volume={415},
   date={2023},
}


\bib{dowerk-thom}{article}{
   author={Dowerk, P. A.},
   author={Thom, A.},
   title={Bounded normal generation and invariant automatic continuity},
   journal={Adv. Math.},
   volume={346},
   date={2019},
   pages={124--169},
}

\bib{dye}{article}{
   author={Dye, H. A.},
   title={On the geometry of projections in certain operator algebras},
   journal={Ann. of Math. (2)},
   volume={61},
   date={1955},
   pages={73--89},
}

\bib{mathoverflow}{misc}{    
    title={What is the outer automorphism group of SU(n)?},    
    author={Figueroa-O'Farrill, J.},    
    note={URL: https://mathoverflow.net/q/40668 (version: 2010-09-30)},    
    eprint={https://mathoverflow.net/q/40668},    
    organization={MathOverflow}  
}


\bib{fulton-harris}{book}{
   author={Fulton, W.},
   author={Harris, J.},
   title={Representation theory},
   series={Graduate Texts in Mathematics},
   volume={129},
   note={A first course;
   Readings in Mathematics},
   publisher={Springer-Verlag, New York},
   date={1991},
   pages={xvi+551},
   isbn={0-387-97527-6},
   isbn={0-387-97495-4},
}




\bib{gardella-thiel}{article}{
author={Gardella E.},
author={Thiel, H.},
title={Rings and C*-algebras generated by commutators},
eprint={https://arxiv.org/abs/2301.05958},
date={2023},

}



\bib{hofmann-morris}{book}{
   author={Hofmann, K. H.},
   author={Morris, S. A.},
   title={The structure of compact groups---a primer for the student---a
   handbook for the expert},
   series={De Gruyter Studies in Mathematics},
   volume={25},
   note={Fourth edition [of  1646190]},
   publisher={De Gruyter, Berlin},
   date={2020},
   pages={xxvii+1006},
   isbn={978-3-11-069599-1},
   isbn={978-3-11-069595-3},
   isbn={978-3-11-069601-1},
}


\bib{hof-morris-survey}{article}{
   author={Hofmann, K. H.},
   author={Morris, S. A.},
   title={Open mapping theorem for topological groups},
   journal={Topology Proc.},
   volume={31},
   date={2007},
   number={2},
   pages={533--551},
}

\bib{kleinecke}{article}{
   author={Kleinecke, D. C.},
   title={On operator commutators},
   journal={Proc. Amer. Math. Soc.},
   volume={8},
   date={1957},
   pages={535--536},
}


\bib{maissen}{article}{
   author={Maissen, B.},
   title={Lie-Gruppen mit Banachr\"{a}umen als Parameterr\"{a}ume},
   language={German},
   journal={Acta Math.},
   volume={108},
   date={1962},
   pages={229--270},
}


\bib{neeb2004}{article}{
   author={Neeb, K. H.},
   title={Infinite-dimensional groups and their representations},
   conference={
      title={Lie theory},
   },
   book={
      series={Progr. Math.},
      volume={228},
      publisher={Birkh\"{a}user Boston, Boston, MA},
   },
   isbn={0-8176-3373-1},
   date={2004},
   pages={213--328},
}


\bib{ng-robert}{article}{
   author={Ng, P. W.},
   author={Robert, L.},
   title={Sums of commutators in pure $\rm C^*$-algebras},
   journal={M\"{u}nster J. Math.},
   volume={9},
   date={2016},
   number={1},
   pages={121--154},
}

\bib{pop}{article}{
   author={Pop, C.},
   title={Finite sums of commutators},
   journal={Proc. Amer. Math. Soc.},
   volume={130},
   date={2002},
   number={10},
   pages={3039--3041},
}


\bib{robert-rordam}{article}{
   author={Robert, L.},
   author={R\o rdam, M.},
   title={Divisibility properties for $C^*$-algebras},
   journal={Proc. Lond. Math. Soc. (3)},
   volume={106},
   date={2013},
   number={6},
   pages={1330--1370},
}



\bib{robertlie}{article}{
   author={Robert, L.},
   title={On the Lie ideals of $C^*$-algebras},
   journal={J. Operator Theory},
   volume={75},
   date={2016},
   number={2},
   pages={387--408},
}




\bib{robertnormal}{article}{
   author={Robert, L.},
   title={Normal subgroups of invertibles and of unitaries in a ${\rm
   C}^*$-algebra},
   journal={J. Reine Angew. Math.},
   volume={756},
   date={2019},
   pages={285--319},
}


\bib{sarkowicz}{article}{
author={Sarkowicz, P.},
title={Unitary groups, K-theory, and traces},
eprint={https://arxiv.org/pdf/2305.15989.pdf},
date={2023}
}

\end{biblist}
\end{bibdiv}
\end{document}